\documentclass[preprint,3p,authoryear]{elsarticle}
\journal{XXX}
\usepackage{amsthm,amssymb,amsmath}
\usepackage{enumitem}
\usepackage{algorithm}
\usepackage{algorithmic}
\usepackage{multirow}
\usepackage{xcolor}
\usepackage{color}

\usepackage{url}
\usepackage{cases}

\newtheorem{theorem}{Theorem}[section]
\newtheorem{definition}{Definition}[section]

\newtheorem{corollary}{Corollary}[section]
\newtheorem{lemma}{Lemma}[section]
\newtheorem{example}{Example}[section]

\newtheorem{remark}{Remark}

\newcommand{\E}{\mathbb{E}}

\begin{document}

\begin{frontmatter}
\title{Some Characterizations and Properties of COM-Poisson Random Variables}

\address[label1]{School of Mathematics and Statistics, Central China Normal University, Wuhan, 430079, China}
\address[label2]{School of Mathematical Sciences, Peking University, Beijing, 100871, China}
\author[label1]{Bo Li}
\ead{haoyoulibo@163.com}
\author[label2]{Huiming Zhang $^{*,}$}
\ead{zhanghuiming@pku.edu.cn}
\author[label1]{Jiao He}
\ead{hejiao2010@hotmail.com}

\fntext[label]{Corresponding author. Bo Li and Huiming Zhang are co-first authors which contributes equally to this work.}

\begin{abstract} This paper introduces some new characterizations of COM-Poisson random variable. First, it extends Moran-Chatterji characterization, and generalizes Rao-Rubin characterization of Poisson distribution to COM-Poisson distribution. Then, it defines the COM-type discrete r.v. ${X_\nu }$ of the discrete random variable $X$. The probability mass function of ${X_\nu }$ has a link to the R{\'e}nyi entropy and Tsallis entropy of order $\nu $ of $X$. And then we can get the characterization of Stam inequality for COM-type discrete version Fisher information. By using the recurrence formula, the property that COM-Poisson random variables~($\nu \ne 1$) is not closed under addition are obtained. Finally, under the property of ``not closed under addition" of COM-Poisson random variables, a new characterization of Poisson distribution is found.
\end{abstract}

\end{frontmatter}
\section{Introduction}
Recently, it is surprising that the Poisson distribution--along with COM-Poisson distribution-- plays a significant role in the development of discrete distribution of fitting count data. The COM-Poisson distribution is a two-parameter extension of the Poisson distribution which contains a wide range of over-dispersion and underdispersion properties. However, the Poisson distribution has some strong ideal assumptions, such as one-parameter and equidispersion. The COM-Poisson distribution was briefly introduced by \cite{conway62} as a model for steady state queuing systems with state-dependent arrival or service rates (in other words, birth-death process with Poisson arrival rate and exponential service rate), see the Appendix.

 \cite{shmueli05} rediscovered this distribution and gave a detailed study of probability and statistics, and the application of COM-Poisson distribution. The term ``Conway-Maxwell-Poisson" was also proposed by \cite{shmueli05}. The probability mass function~(p.m.f.) is given by
\begin{equation} \label{eq:com}
P(X = k) = \frac{{{\lambda ^k}}}{{{{(k!)}^\nu }}} \cdot \frac{1}{{Z(\lambda ,\nu )}},(k = 0,1,2, \dots ),
\end{equation}
where $\lambda ,\nu  > 0$ and $Z(\lambda ,\nu ) = \sum\limits_{i = 0}^\infty  {\frac{{{\lambda ^i}}}{{{{(i!)}^\nu }}}} $. We denote (\ref{eq:com}) as $X \sim {\rm{CMP}}(\lambda ,\nu )$.

Meanwhile, some other theoretical properties of Conway-Maxwell-Poisson distribution are also springing up in recent years. \cite{kokonendji08} showed that COM-Poisson distribution was overdispersed when $\nu  \in {\rm{[0,1)}}$ and underdispersed when $\nu  \in {\rm{(1, + }}\infty {\rm{)}}$. Assumed $\nu$ to be an integer, \cite{nadarajah09} derived explicit expressions for its moments and the cumulative distribution function.  \cite{shmueli05} produced an approximation for normalization constant $Z(\lambda ,\nu )$ for integer values, they conjectured that it was also valid for non-integers. \cite{gillispie15} proved the conjecture to be true. The approximation is
\begin{equation} \label{eq:app}
Z(\lambda ,\nu ) = \frac{{\exp (\nu {\lambda ^{1/\nu }})}}{{{\lambda ^{(\nu  - 1)/2\nu }}{{(2\pi )}^{(\nu  - 1)/2}}\sqrt \nu  }}[1 + O(\frac{1}{{{\lambda ^{1/\nu }}}})],
\end{equation}
for all fixed $\nu  > 0$ as $\lambda  \to \infty $ (equivalently, ${\lambda ^{1/\nu }} \to \infty $).

 \cite{borges14} gave conditions that COM-binomial random variable~(r.v.) converges in distribution to a COM-Poisson r.v.~(see Definition 2.1 or Example 3.2 below). \cite{brown01} considered a very large class of approximation distribution which is the equilibrium distribution of a birth-death process with arrival rate ${\alpha _i}$ and service rate ${\beta _i}$, and got the Stein's identities (the functional operator characterizations for this ``very large class of distribution"). By applying the Stein's identities from \cite{brown01},  \cite{daly16} gave an explicit bound in total variation distance between the COM-binomial distribution and the corresponding COM-Poisson limit. \cite{pogany16} found an integral expression for the COM-Poisson normalizing constant $Z(\lambda ,\nu )$.

The aim of this work is to illustrate some theoretical properties of COM-Poisson distribution, especially the characterizations of COM-Poisson distribution. Except density function and generation function, it is an interesting topic to find other ``iff" condition of certain distribution. Some monographs and sections of monographs on many characterizations of discrete distribution have been published on it, see \cite{kagan73}, \cite{patil75}. Interestingly, there are much researches about the Poisson distribution and its related distribution, see \cite{johnson05}, \cite{haight67} and the reference therein.

\section{Characterization by COM-binomial distribution}

In this section, we will deduce the conditional distribution of the COM-Poisson distribution which is useful in following section. Consider the sum of two independent COM-Poisson r.v.'s with parameters $(\lambda_{1},\nu)$ and $(\lambda_{2},\nu)$: $S=X+Y$, then we have
\begin{align*}\label{eq:first}
P(S = s) &= \sum\limits_{x = 0}^s {P(X = x)} P(Y = s - x)\\
&= \sum\limits_{x = 0}^s \frac{\lambda_{1}^{x}}{(x!)^{\nu}}\frac{1}{Z(\lambda_{1},\nu)} \frac{\lambda_{2}^{s-x}}{((s-x)!)^{\nu}} \frac{1}{Z(\lambda_{2},\nu)}\\
&= \frac{1}{Z(\lambda_{1},\nu)}\frac{1}{Z(\lambda_{2},\nu)}
 \sum\limits_{x = 0}^s
 \frac{ (s!)^{\nu} (\lambda_{1}+\lambda_{2})^{s} \lambda_{1}^{x}\lambda_{2}^{s-x}}
 { (s!)^{\nu} (x!)^{\nu} ((s-x)!)^{\nu} (\lambda_{1}+\lambda_{2})^{s}}.\\
\end{align*}
Rewrite the above expression as a binomial-like summation:
\begin{equation}\label{eq:first}
\frac{(\lambda_{1}+\lambda_{2})^{s}}{(s!)^{\nu}Z(\lambda_{1},\nu) Z(\lambda_{2},\nu)}
 \sum\limits_{x = 0}^s {{s} \choose {x}}^{\nu} \left(\frac{\lambda_{1}}{\lambda_{1}+\lambda_{2}}\right)^{x} \left(\frac{\lambda_{2}}{\lambda_{1}+\lambda_{2}}\right)^{s-x}.
\end{equation}
The conditional distribution $P(X = k\left| {S = s} \right.)$ is

\begin{equation} \label{eq:cnh}
 \frac{{P(X = k)P(Y = s - k)}}{{P(S = s)}} = {{s} \choose {k}}^{\nu} \left(\frac{\lambda_{1}}{\lambda_{1}+\lambda_{2}}\right)^{k} \left(\frac{\lambda_{2}}{\lambda_{1}+\lambda_{2}}\right)^{s-k} \bigg/ \sum\limits_{x = 0}^s {{s} \choose {x}}^{\nu} \left(\frac{\lambda_{1}}{\lambda_{1}+\lambda_{2}}\right)^{x} \left(\frac{\lambda_{2}}{\lambda_{1}+\lambda_{2}}\right)^{s-x}.
\end{equation}

It is easy to define the p.m.f. of COM-binomal distribution with parameters $\nu$, $m$ and $p$, see \cite{shmueli05}, \cite{borges14}:
\begin{definition}
The COM-binomial distribution~(CMB) is a distribution with p.m.f.:
\begin{equation} \label{eq:cb}
P(X = k) = \frac{{{{{{m} \choose {k}}}^\nu}{p^k}{{(1 - p)}^{m - k}}}}{{\sum\limits_{i = 0}^m {{{{{m} \choose {i}}}^\nu}{p^i}{{(1 - p)}^{m - i}}} }}
 = \frac{ {{{{{m} \choose {k}}}^\nu}{p^k}{{(1 - p)}^{m - k}}} }{N(m,p,\nu)}
,\quad k = 0,1, \dots ,m.
\end{equation}
where $\nu  \in {{\rm{R}}^ + }, \  m \in \Bbb{N}=:\{0,1,2, \dots\}, \  p \in (0,1)$, and $N(m,p,\nu)=: \sum\limits_{i = 0}^m {{{{{m} \choose {i}}}^\nu}{p^i}{{(1 - p)}^{m - i}}}$. We denote (\ref{eq:cb}) as $ X \sim {\rm{CMB}}(m, p,\nu )$.
\end{definition}

Then if $X$ and $Y$ are independent random variables, we have $X\left| {{\rm{ }}X + Y} \right. \sim {\rm{CMB}}(m,\frac{{{\lambda _1}}}{{{\lambda _1} + {\lambda _2}}},\nu )$ for $X \sim {\rm{CMP}}(\lambda_1 ,\nu )$ and $Y \sim {\rm{CMP}}(\lambda_1 ,\nu )$.
\cite{moran52} and \cite{chatterji63} proved that if $X,Y$ are independent discrete r.v.'s, and $p$ is a constant, then

$$
  P(X=x \mid X+Y=s ) = {{{s} \choose {x}}}{p^x}{{(1 - p)}^{s-x}}, \: x = 0,1,...,s.
$$
holds iff $X$ and $Y$ each has Poisson distribution with parameters in the ratio $p:p-1$.

Motivating by \cite{patil64}'s general results for conditional distribution characterization Poisson distribution, we have a theorem for characterizing COM-Poisson distribution, see also \cite{kagan73}. The method of the proof is the same as \cite{chatterji63}.

\begin{lemma}
Suppose $P(X \left| {X+Y } \right.)$ is the function $c(x,x + y)$ for independent r.v.'s $X$ with p.m.f. $f(x) > 0$ and $Y$ with p.m.f. $g(y) > 0$~(both discrete or both continuous), such that $\frac{{c(x + y,x + y)c(0,y)}}{{c(x,x + y)c(y,y)}}$ is the form of $h(x + y)/h(x)h(y)$ where $h( \cdot )$ is an arbitrary non-negative function, then
\begin{equation} \label{eq:fg}
f(x) = f(0)h(x){e^{ax}},g(y) = g(0)\frac{{h(y)c(0,y)}}{{c(y,y)}}{e^{ay}},
\end{equation}
where $f(0)$ and $g(0)$ are the corresponding normalizer for $f(x)$ and $g(y)$ respectively, which are the p.m.f..
\end{lemma}

\begin{theorem}
Let $X$ and $Y$ be independent discrete r.v.'s with p.m.f. $f(x) > 0$ and $g(y) > 0$, respectively. If $P(X=x \left| {X+Y=x+y } \right.)$ is the COM-binomial distribution with parameter $(x+y ,\frac{1}{{\theta  + 1}}, \nu)$ in (\ref{eq:cb}), then $X \sim {\rm{CMP}}(\lambda , \nu )$,~$\theta  \in (0,\infty )$, and $Y \sim {\rm{CMP}}~(\theta \lambda , \nu )$,~($\lambda > 0$).
\end{theorem}

\begin{proof}
We already know that $c(x,x + y) = {{x+y} \choose {x}}^{\nu} {\left( {\frac{1}{{\theta  + 1}}} \right)^x}{\left( {\frac{\theta }{{\theta  + 1}}} \right)^{(x + y)-x}}/{N(x+y,\frac{1}{1+\theta},\nu)}$,
so we have
\begin{equation}
  c( a,b) = {{b} \choose {a}}^{\nu} {\left( {\frac{1}{\theta }} \right)^a}{\left( {\frac{\theta }{{1 + \theta }}} \right)^b}/N(b,\frac{1}{{1 + \theta }},\nu ).
\end{equation}
by noticing that ${\left( {\frac{1}{{1 + \theta }}} \right)^a}{\left( {\frac{\theta }{{1 + \theta }}} \right)^{b - a}} = {\left( {\frac{1}{\theta }} \right)^a}{\left( {\frac{\theta }{{1 + \theta }}} \right)^b}$.
Hence,
 \begin{align*}
\frac{{c(x + y,x + y) \cdot c(0,y)}}{{c(x,x + y)\cdot c(y,y)}}
&= \frac{\left(\frac{1}{\theta}\right)^{x+y}\left(\frac{\theta}{1+\theta}\right)^{x+y} \cdot \left(\frac{\theta}{1+\theta}\right)^{y} }
{{{x+y} \choose {x}}^{\nu} \left(\frac{1}{\theta}\right)^{x}\left(\frac{\theta}{1+\theta}\right)^{x+y}\cdot \left(\frac{1}{\theta}\right)^{y} \left(\frac{\theta}{1+\theta}\right)^{y} }\\
 &= \left( \frac{1}{(x+y)!}/\frac{1}{x!} \frac{1}{y!}\right)^{\nu}.
 \end{align*}
Then we have $h(x) = \left( \frac{1}{x!} \right)^{\nu}$. In view of (\ref{eq:fg}), $\frac{c(0,y)}{c(y,y)}= \theta ^{y} $, let $\lambda= {e^a}$ and compared with the expression (\ref{eq:com}), hence
  $$f(x) = f(0)\frac{{{\lambda ^x}}}{{{{(x!)}^\nu }}} , g(y) = g(0) \frac{ (\theta \lambda) ^ {y}}{(y!)^{\nu}},$$
are p.m.f. of COM-Poisson distributions.
\end{proof}

The following Rao-Rubin characterization for Poisson law is based on a interesting model
where an original observation from a discrete distribution is subject to
damage according to a binomial distribution. Let $X,Y$ be two discrete r.v.'s, \cite{rao64} proved that if
$$
  P(Y=r \mid  X= n) = {{{{{n} \choose {r}}}}{p^r}{{(1 - p)}^{n-r}}},\: p \in (0,1),
$$
then
$$
  P(Y=r) = P(Y=r \mid  X= Y) , r=0,1,\dots.
$$
iff $X$ is Poisson distributed, the simple proof can be found in \cite{wang75}.

Based on the above-mentioned COM-binomial distribution and an extension of \cite{rao64}'s characterization which was constructed by \cite{shanbhag77}, we obtain the Rao-Rubin characterization for COM-Poisson distribution.

\begin{lemma}\label{lem:shan}
Let $X,Y$ be the non-negative r.v.'s such that $P(X = z) = {P_z}$ with ${P_0} < 1,{P_z} > 0, z \in \Bbb{N}$, and
 \begin{align*}
P(\left. {Y = r} \right|X = z) = \frac{{{a_r}{b_{z - r}}}}{{\sum\nolimits_{s = 0}^z  {{a_s}{b_{z - s}}} }} = :\frac{{{a_r}{b_{z - r}}}}{{{c_z}}},(r = 0,1, \dots ,z),
 \end{align*}
 where ${a_z} > 0$ for all $z \ge 0$, ${b_0},{b_1} > 0$ and ${b_z} \ge 0$ for $z \ge 2$, then
$$
  P(Y = r) = P(\left. {Y = r} \right|X = Y),(r = 0,1, \dots ) .
$$
iff
\begin{equation}\label{eq:second}
  \frac{{{P_z}}}{{{c_z}}} = \frac{{{P_0}}}{{{c_0}}}{\theta ^z},(z = 0,1, \dots )  \quad {\rm{for \ some}}\quad \theta  > 0 .
\end{equation}
\end{lemma}

Next, we will give the following COM-Poisson extension of Rao-Rubin characterization. The result of \cite{shanbhag77} will be useful to get a proof immediately.

\begin{theorem}
Let $X,Y$ be the discrete r.v.'s such that $P(X = z) = {P_z}$ with ${P_0} < 1,{P_z} > 0, z \in \Bbb{N}$, and
 \begin{align*}
P(\left. {Y = r} \right|X = z) = {{z} \choose {k}}^{\nu} \left(\frac{m}{m+n}\right)^{k} \left(\frac{n}{m+n}\right)^{z-k}/
\sum\limits_{x = 0}^z {{z} \choose {x}}^{\nu} \left(\frac{m}{m+n}\right)^{x} \left(\frac{n}{m+n}\right)^{z-x},(m,n \in \Bbb{N})
 \end{align*}
then $P(Y = r) = P(\left. {Y = r} \right|X =Y),(r = 0,1, \dots ) $ iff $X \sim {\rm{CMP}} (\theta ,\nu)$ for some $\theta > 0$ .
\end{theorem}

\begin{proof}

From the normalizer in (\ref{eq:cb}), we have
\begin{align*}
P(\left. {Y = k} \right|X = z) &= \frac{{{z} \choose {k}}^{\nu} \left(\frac{m}{m+n}\right)^{k} \left(\frac{n}{m+n}\right)^{z-k}}{N(m, \frac{m}{m+n}, \nu)}\\
&= \frac{{\frac{1}{{{{(k!)}^\nu }}}{{(\frac{m}{{m + n}})}^k}}}{{\sqrt {N(z,\frac{m}{{m + n}},\nu )} }} \cdot \frac{{\frac{1}{{{{((z - k)!)}^\nu }}}{{(\frac{n}{{m + n}})}^{z - k}}}}{{\sqrt {N(z,\frac{m}{{m + n}},\nu )} }}/{(z!)^{ - \nu }}
 = :\frac{{{a_k}{b_{z - k}}}}{{{c_z}}}
\end{align*}
where ${a_k},{b_{z-k}},{c_z}$ satisfy the conditions in \textbf{Lemma \ref{lem:shan}}. Compared with (\ref{eq:second}), we have ${c_z} = {(z!)^{-\nu }}$. This implies that ${P_z} = \frac{{{P_0}{\theta ^z}}}{{{{(z!)}^\nu }}}$. And summing $z$ over $z = 0,1, \ldots $, we have
\[{P_0} = {\left( {\sum\limits_{z = 0}^\infty  {\frac{{{\theta ^z}}}{{{{(z!)}^\nu }}}} } \right)^{ - 1}}\]
The last expression has $X \sim {\rm{CMP}} (\theta, \nu)$ for some $\theta > 0$.
\end{proof}

\section{COM-type distribution and discrete version Fisher information}

\subsection{COM-type distribution}
In this section, COM-Poisson can be characterized by the equality sign in a discrete version of the Stam inequality~(\cite{stam59}) for the COM-type Fisher information based on Kagan's characterization~(\cite{kagan01}) of Poisson distribution.

For any discrete distribution (or r.v.), we first define its COM-type distribution (or r.v.).
\begin{definition}\label{def:cmt}
Let $P_{X}(x)=P(X=x)$ be the p.m.f. of a discrete r.v. $X$, a discrete r.v.$X_{\nu}$ is said to have a COM-type of order $\nu$ of the distribution  (if exists, namely $\sum_{x=1}^{\infty} P_{X}^{\nu}(x) < \infty$), if the p.m.f. of $X_{\nu}$ is given by
\begin{equation} \label{eq:four}
P_{X_{\nu}}(x)= \frac{P_{X}^{\nu}(x)}{\sum_{x=0}^{\infty} P_{X}^{\nu}(x)} =:{C_{{X_\nu }}} P_{X}^{\nu}(x),  (\nu >0)
\end{equation}
with ${C_{{X_\nu }}} = 1/\sum\nolimits_{x = 0}^\infty  {P_X^\nu } (x)$.
\end{definition}
By the structure of COM-type distribution, we call it ``$\nu$-power distribution" with generating p.m.f..

Next, we illustrate some examples of COM-type distributions.

\begin{example}
For the COM-type Poisson distribution connected to (\ref{eq:com}), the p.m.f. is given by
$$
  P_{X_{\nu}}(x)=\frac{\lambda^{x}}{(x!)^{\nu}}Z^{-1}(\lambda,\nu)= \left(\frac{\mu^{x}}{x!}\right)^{\nu}Z^{-1}(\mu^{\nu},\nu),\:(let \  \lambda =\mu^{\nu}).
$$
where $\lambda ,\nu  > 0$.
\end{example}

\begin{example}
The p.m.f. of COM-type binomial distribution linked to (\ref{eq:cb}) is given by
$$
  P_{X_{\nu}}(x)= \frac{{{n} \choose {x}}^{\nu} p^{x}(1-p)^{n-x}}{\sum\limits_{k = 0}^n {{n} \choose {k}}^{\nu} p^{k}(1-p)^{n-k}}
  =: \frac{\bigg({{n} \choose {x}} q^{x}(1-q)^{n-x}\bigg)^{\nu}}{\sum\limits_{k = 0}^n \bigg({{n} \choose {k}} q^{k}(1-q)^{n-k}\bigg)^{\nu}}, \: (let \ \left(\frac{p}{1-p}\right)^{\frac{1}{\nu}}=\frac{q}{1-q}).
$$
where $\nu  \in {{\rm{R}}^ + }, \  n \in \Bbb{N}, \  p \in (0,1)$.
\end{example}

\begin{example}
A r.v. $X$ is said to follow COM-negative binomial distribution $({\rm{CMNB}}(r,\nu ,p))$ with three parameters $(r,\nu ,p)$ (see \cite{zhang18}), if the p.m.f. is given by
\begin{equation} \label{eq:cnbd}
{\rm{P}}(X = k) = \frac{{{{\big(\frac{{\Gamma (r + k)}}{{k!{\mkern 1mu} \Gamma (r)}}\big)}^\nu }{p^k}{{(1 - p)}^r}}}{{\sum\limits_{i = 0}^\infty  {{{\big(\frac{{\Gamma (r + i)}}{{i!{\mkern 1mu} \Gamma (r)}}\big)}^\nu }} {p^i}{{(1 - p)}^r}}} = :{{\bigg(\frac{{\Gamma (r + k)}}{{k!{\mkern 1mu} \Gamma (r)}}\bigg)}^\nu }\frac{{{p^k}{{(1 - p)}^r}}}{{C(r,\nu ,p)}},\quad (k = 0,1,2, \ldots ),
\end{equation}
where $r,v \in (0,\infty )$ and $p \in (0,1)$. We denote (\ref{eq:cnbd}) as $X \sim {\rm{CMNB}}(r,\nu ,p)$. Then we have
\[{P_{{X_\nu }}}(x) = {\bigg(\frac{{\Gamma (r + x)}}{{x!\Gamma (r)}}{p^{x/\nu }}{(1 - p)^{r/\nu }}\bigg)^\nu }\bigg/\sum\limits_{i = 0}^\infty  {{{\bigg(\frac{{\Gamma (r + i)}}{{i!\Gamma (r)}}{p^{i/\nu }}{{(1 - p)}^{r/\nu }}\bigg)}^\nu }}.\]
\end{example}

\begin{example}
As we know, the geometric distribution is ${{P_X}(x) = p{{(1 - p)}^x}}$, and the COM-type is
\begin{equation}
  {P_{{X_\nu }}}(x) = \frac{{{p^\nu }{{(1 - p)}^{\nu x}}}}{{\sum\limits_{i = 0}^\infty  {{p^\nu }{{(1 - p)}^{\nu i}}} }} = \left[1 - {(1 - p)}^\nu \right](1 - p)^{\nu x},p \in (0,1),
\end{equation}
so it is easy to see that the COM-type of COM-geometric distribution maintains a geometric distribution.
\end{example}

\begin{example}
The Riemann zeta distribution~(see \cite{lin08}) with p.m.f.
\[P(X = x) = \frac{{{x^{ - \sigma }}}}{{\sum\nolimits_{i = 1}^\infty  {{i^{ - \sigma }}} }} =: \frac{{{x^{ - \sigma }}}}{{\xi (\sigma )}},(x = 1,2, \dots ;\sigma  > 1),\]
then the COM-type Riemann zeta distribution:
\[{P_{{X_\nu }}}(x) = \frac{{{x^{ - \nu \sigma }}}}{{\sum\nolimits_{i = 1}^\infty  {{i^{ - \nu \sigma }}} }} =: \frac{{{x^{ - \nu \sigma }}}}{{\xi (\nu \sigma )}},(x = 1,2, \dots ;\nu \sigma  > 1)\]
is also a Riemann zeta distribution if $\nu  \ge 1$.
\end{example}

\begin{example}
Consider a member of the Lerch-type distributions with probability generating function~(p.g.f.) :
\begin{equation} \label{eq:Lerch}
G(z) = \frac{{\Phi (\rho z,1,c)}}{{\Phi (\rho ,1,c)}} \buildrel \Delta \over = {{\sum\limits_{i = 0}^\infty  {\frac{{{\rho ^i}}}{{(c + i)}} \cdot {z^i}} } \mathord{\left/
 {\vphantom {{\sum\limits_{i = 0}^\infty  {\frac{{{\rho ^i}}}{{(c + i)}} \cdot {z^i}} } {\sum\limits_{i = 0}^\infty  {\frac{{{\rho ^i}}}{{(c + i)}} \cdot {z^i}} }}} \right.
 \kern-\nulldelimiterspace} {\sum\limits_{i = 0}^\infty  {\frac{{{\rho ^i}}}{{(c + i)}}} }},(0 < \rho  < 1,c > 0)
\end{equation}
where $\Phi (\rho ,v,c) = \sum\limits_{i = 0}^\infty  {\frac{{{\rho ^i}}}{{{{(c + i)}^v}}} } $ is Lerch¡¯s transcendent function. See \cite{johnson05}, page 526.

Let $(X_\nu,\Theta )$ be  such a bivariate r.v. that their joint distribution is
\[f(x,\theta ) = C\frac{\theta ^{\nu  - 1}}{x!} \exp \{ ax - (d + bx)\theta \} ,(x = 0,1,2 \dots )\]
where $C$ is a normalisation constant.
\cite{gomez14} have found that the marginal distribution of $X$ is the COM-type of a Lerch-type distribution (\ref{eq:Lerch})
\[P({X_\nu } = x) = {{\frac{{{e^{ax}}}}{{{{(d + bx)}^\nu }}}} \frac{b^\nu}{\Phi \big({e^a},\nu,\frac{d}{b}\big)}},(d,b,\nu, \in {\Bbb{R}},a < 0).\]
\end{example}

\begin{example}
COM-Hyper-Poisson distribution (Shifted COM-Poisson distribution, \cite{ahmad07}) with p.m.f.
\[{P_{{X_\nu }}}(x) = {{{{\left( {\frac{{{\lambda ^x}}}{{(a + x)!}}} \right)}^\nu}} \mathord{\left/
 {\vphantom {{{{\left( {\frac{{{\lambda ^x}}}{{(a + x)!}}} \right)}^\nu}} {\sum\limits_{i = 0}^\infty  {{{\left( {\frac{{{\lambda ^i}}}{{(a + i)!}}} \right)}^\nu}} }}} \right.
 \kern-\nulldelimiterspace} {\sum\limits_{i = 0}^\infty  {{{\left( {\frac{{{\lambda ^i}}}{{(a + i)!}}} \right)}^\nu}} }},(a \ge 0,\lambda  > 0).\]
\end{example}

\begin{remark}
\cite{chakraborty15} considered the extended COM-Poisson distribution (ECOMP$(r,\theta ,\alpha ,\beta )$):
\begin{equation} \label{eq:E}
\mathrm{P}(X = k) = \frac{{\Gamma {{(r + k)}^\beta }}}{{{{(k!)}^{\alpha }}}}{\theta ^k}/\sum\limits_{i = 1}^\infty  {\frac{{\Gamma {{(r + i)}^\beta }}}{{{{(i!)}^{\alpha  }}}}{\theta ^i}} \quad (k = 0,1,2, \ldots ),
\end{equation}
where the parameter space is $(r \ge 0,\theta  > 0,\alpha  > \beta ) \cup (r > 0,0 < \theta  < 1,\alpha  = \beta )$. COM-negative binomial distribution is a special case of ECOMP$(r,\theta ,\alpha ,\beta )$ when $\alpha   = 1 ,\beta  = \nu $.
Another generation of COM-type distribution was obtained by \cite{imoto14} (see also \cite{chakraborty15}):
$$
{\rm{P}}(X = x) = \frac{{\Gamma {{(r + x)}^\nu }}}{{{{x!}}}}{p^x}/\sum\limits_{i = 1}^\infty  {\frac{{\Gamma {{(r + i)}^\nu }}}{{{{i!}}}}{p ^i}} \quad (r,v > 0, p \in (0,1))
$$
which includes the negative binomial distribution, but not a COM-Poisson type distribution of negative binomial distribution by Definition \ref{def:cmt}.
\end{remark}

\begin{remark}
Recall the R{\'e}nyi entropy (see \cite{renyi61}) in the information theory, which generalizes the Shannon entropy. The R{\'e}nyi entropy of order $\alpha$ of a discrete r.v. $X$:
$$
H_\alpha ^R(X) = \frac{1}{{1 - \alpha }}{\rm{ln}}\sum\limits_{i = 0}^\infty  {[P} (X = i){]^\alpha },(\alpha  \ne 1).
$$

Then the normalization constant ${C_{{X_\alpha }}}$ in (\ref{eq:four}) has R{\'e}nyi entropy representation ${C_{{X_\alpha }}} = {e^{(\alpha  - 1)H_\alpha ^R(X)}}$,
so ${P_{{X_\nu }}}(x) = P_X^\nu (x){e^{(\alpha  - 1)H_\alpha ^R(X)}}$.

In physics, another generalization of Shannon entropy is the Tsallis entropy, Tsallis entropy of order $\alpha$ of a discrete r.v. $X$ is defined by
$$
H_\alpha ^T(X) = \frac{1}{{1 - \alpha }}\left( {\sum\limits_{i = 0}^\infty  {[P} (X = i){]^\alpha } - 1} \right),(\alpha  \ne 1).
$$
This entropy was introduced by \cite{tsallis88} as a basis for generalizing the Boltzmann-Gibbs statistics.

Hence the normalization constant ${C_{{X_\alpha }}}$ has Tsallis entropy representation ${C_{{X_\alpha }}} = {[1 + (1 - \alpha )H_\alpha ^T(X)]^{ - 1}}$,
then ${P_{{X_\nu }}}(x) = P_X^\nu (x){[1 + (1 - \alpha )H_\alpha ^T(X)]^{ - 1}}$.
\end{remark}

The next result shows that the COM-type of order $\frac{1}{\nu}$ distribution related to $X_{\nu}$ is $X$. Hence, the COM-type distribution has one-to-one correspondence between $X_{\nu}$ and $X$.

\begin{lemma}
$P_{(X_{\nu})_{\frac{1}{\nu}}}(\cdot)= P_{X}(\cdot)$ for any $\nu >0$, then we have $(X_{\nu})_{\frac{1}{\nu}}=X$ in distribution.
\end{lemma}
\begin{proof}
 In fact,
 $$P_{(X_{\nu})_{\frac{1}{\nu}}}(x) =
 \left[{C_{{X_\nu }}} P_{X}^{\nu}(x) \right]^{\frac{1}{\nu}}/ \sum\limits_{x = 0}^{\infty} [{C_{{X_\nu }}} P_{X}^{\nu}(x)]^{\frac{1}{\nu}}=P_{X}(x).$$
\end{proof}
Let $\E_{\nu}X$  be the expectation of its COM-type r.v., that is
\begin{equation}
   \E_{\nu}X =: \E X_{\nu} = \sum\limits_{x = 0}^{\infty} x P_{X_{\nu}}(x),
\end{equation}
then we have $\E_{\nu}f(X) = \sum\limits_{x = 0}^{\infty} f(x) P_{X_{\nu}}(x)$.

The following lemma is an immediate consequence of Lemma 3.1.

\begin{lemma}
$\E _{\frac{1}{\nu}} X_{\nu} = \E X $ and $\E _{\frac{1}{\nu}} f( X_{\nu}) = \E f (X)$  for a measurable function $f$.
\end{lemma}

\subsection{Kagan's characterization}
Based on the two lemmas, it would facilitate discussing COM-type discrete Fisher information. For a discrete r.v. $X$ with values in $\Bbb{N}$, \cite{kagan01} defined

$$
J_{X}(x)=
\left\{
  \begin{array}{ll}
  \displaystyle 1- \frac{P_{X}(x-1)}{P_{X}(x)},\; if \ P_{X}(x)>0,
  \\
  \\
  \displaystyle 0,\; if~P_{X}(x)=0
   \end{array}
   \right.
$$
and $I_{X}=\E J_{X}^2$ which is a discrete version of the Fisher information.

Now, we consider the COM version of Fisher information, which allow us to apply Kagan's characterization of Poisson distribution. Let
$$
K_{X}(x)=
\left\{
  \begin{array}{ll}
  \displaystyle 1-\left[ \frac{P_{X}(x-1)}{P_{X}(x)}\right]^{\frac{1}{\nu}},\; if~P_{X}(x)>0
  \\
  \\
  \displaystyle 0,\; if~P_{X}(x)=0
   \end{array}
   \right.
$$

Then COM-type discrete Fisher information can be defined by

\begin{equation}\label{eq:cvi}
{{\rm{C}}_\nu }{{\rm{I}}_X}=\E _{\frac{1}{\nu}} \left[K_{X}^2\right] = \sum\limits_{x = 0}^\infty  {\left( {1 - {{\left[ {\frac{{{P_X}(x - 1)}}{{{P_X}(x)}}} \right]}^{\frac{1}{\nu }}}} \right)^2} {P_{{X_{\frac{1}{\nu }}}}}(x) = \sum\limits_{x = 0}^\infty  {\left( {1 - \frac{{{P_{{X_{\frac{1}{\nu }}}}}(x - 1)}}{{{P_{{X_{\frac{1}{\nu }}}}}(x)}}} \right)^2} {P_{{X_{\frac{1}{\nu }}}}}(x) = \E  \left[J _{X_{\frac{1}{\nu}}}^2\right].
\end{equation}

We say that $X\in RSP$ (right side positive), i.e. if $P_{X}(x)>0$ then we have $P_{X}(x+1)>0 $. Lemma 2
 in \cite{kagan01} showed that if $X,Y \in RSP$ are independent r.v.'s., it is trivial to see that $X+Y \in RSP$.
\begin{lemma}(\cite{kagan01})
If $I_{X},I_{Y}< \infty$ and $ X,Y \in RSP$, then
\begin{equation}\label{eq:kagan}
\frac{1}{I_{Z}} \geq  \frac{1}{I_{X}}+\frac{1}{I_{Y}},\, (Z=X+Y)
\end{equation}
with the equality sign holding iff $X,Y,Z$ have Poisson distributions (possibly shifted), i.e.
$$
P_{Z}(z)=e^{-\lambda} \frac{\lambda ^{z-z_{0}}}{(z-z_{0})!},(z=z_{0},z_{0}+1, \dots)
$$
for some integer $z_{0}$ and $\lambda > 0$, with the same expression for $P_{X}(x),P_{Y}(y)$.
\end{lemma}

However, for the COM-Poisson distribution of order $\nu$, the convolution of two independent COM-Poisson r.v.s. $X$ and $Y$ may not be the COM-Poisson r.v.'s of order $\nu$ (except $\nu=1$). In view of (\ref{eq:first}), if $X \sim \rm{CMP} ({\lambda _1} ,\nu )$ and $Y \sim \rm{CMP} ({\lambda _2} ,\nu )$ where ${\lambda _1} = \lambda _x^\nu ,{\lambda _2} = \lambda _y^\nu$, we have
\begin{equation} \label{eq:eight}
 P(X+Y=s) =
 \left(\frac{(\lambda_{1}+\lambda_{2})^{s}}{s!}\right)^{\nu}
 \frac{1}{Z(\lambda_{1},\nu) Z(\lambda_{2},\nu)}
 \sum\limits_{k = 0}^s
 \left( \frac{s!}{(s-k)!k!}
 \left(\frac{\lambda_{1}}{\lambda_{1}+\lambda_{2}}\right)^{k} \left(\frac{\lambda_{2}}{\lambda_{1}+\lambda_{2}}\right)^{s-k}
 \right)^{\nu}.
\end{equation}

In order to verify whether (\ref{eq:eight}) is COM-Poisson distributed, we need recurrence relation characterization of COM-Poisson distribution. The following lemma is a property of proportion to characterize COM-Poisson distribution, see also \cite{ahmad07} as a special case of COM-Hyper-Poisson distribution.

\begin{lemma} \label{lem:rec}
If $P(X=n)$ is the p.m.f. of any discrete r.v. $X$, then $X$ is COM-Poisson distribution iff
\begin{equation} \label{eq:rec}
\frac{{P(X = n)}}{{P(X = n - 1)}} = \frac{\lambda }{{{n^\nu }}} \buildrel \Delta \over = {(\frac{\mu }{n})^\nu }, (\nu ,\lambda  > 0,\mu  = {\lambda ^{1/v}}).
\end{equation}
\end{lemma}

\begin{proof}
  It can be derived by the p.m.f. of COM-Poisson distribution (\ref{eq:com}), since
$$
  P(X=n) =\frac{P(X=n)}{P(X=n-1)} \cdot \frac{P(X=n-1)}{P(X=n-2)} \cdot \cdot \cdot \cdot \frac{P(X=1)}{P(X=0)} \cdot P(X=0)
  = \frac{\lambda ^{n}}{(n!)^{\nu}} P(X=0),
$$
$$\sum\limits_{n = 0}^{\infty} P(X=n)=1 \Rightarrow P(X=0)=Z^{-1}(\lambda, \nu)> 0.$$
\end{proof}

Let $ {\mathcal{F_{\nu}}}$ be the family COM-Poison distribution of order $\nu$. It is closed under addition~($X,Y \in {\mathcal{F_{\nu}}} \Rightarrow X + Y \in {\mathcal{F_{\nu}}}$) except when $\nu=1$.

\begin{corollary} \label{cor:cvlt}
If $\nu  \ne 1$, $X \sim {\rm{CMP}}(\lambda_{1},\nu)$ and $Y \sim {\rm{CMP}}(\lambda_{2},\nu)$, then $X+Y \not  \sim   {\rm{CMP}}(\lambda_{1}+\lambda_{2},\nu)$.
\end{corollary}

\begin{proof}
Apply Lemma \ref{lem:rec} to (\ref{eq:eight}), and let ${\lambda _1} = \lambda _x^\nu ,{\lambda _2} = \lambda _y^\nu $, we have
\begin{equation} \label{eq:six}
\frac{{P(X + Y = n)}}{{P(X + Y = n - 1)}} = {\left( {\frac{{{\lambda _x} + {\lambda _y}}}{n}} \right)^\nu } \cdot \frac{{\sum\limits_{k = 0}^n {{{\left( {\frac{{n!}}{{(n - k)!k!}}{{\left( {\frac{{{\lambda _x}}}{{{\lambda _x} + {\lambda _y}}}} \right)}^k}{{\left( {\frac{{{\lambda _y}}}{{{\lambda _x} + {\lambda _y}}}} \right)}^{n - k}}} \right)}^\nu }} }}{{\sum\limits_{k = 0}^{n - 1} {{{\left( {\frac{{(n - 1)!}}{{(n - 1 - k)!k!}}{{\left( {\frac{{{\lambda _x}}}{{{\lambda _x} + {\lambda _y}}}} \right)}^k}{{\left( {\frac{{{\lambda _y}}}{{{\lambda _x} + {\lambda _y}}}} \right)}^{n - 1 - k}}} \right)}^\nu }} }}{\rm{ }} \buildrel \Delta \over = {\left( {\frac{{{\lambda _x} + {\lambda _y}}}{n}} \right)^\nu }\frac{{{a_n}({\lambda _x},{\lambda _y},\nu )}}{{{a_{n - 1}}({\lambda _x},{\lambda _y},\nu )}},
\end{equation}
for $n = 1,2, \dots $.

We observe that the expression (\ref{eq:six}) is slightly different from the (\ref{eq:rec}).

Compared (\ref{eq:rec}) with (\ref{eq:six}), if  $X+Y \sim  {\rm{CMP}}(\lambda_{1}+\lambda_{2},\nu)$, we must have ${{a_n}({\lambda _x},{\lambda _y},\nu ) = {a_{n - 1}}({\lambda _x},{\lambda _y},\nu )}$ for $n = 1,2, \dots $. So ${a_n}({\lambda _x},{\lambda _y},\nu ) \equiv {a_0}({\lambda _x},{\lambda _y},\nu ) = 1$ for $n = 1,2, \dots $ and $\nu  \in {{\Bbb{R}}^ + } \backslash \{ 1\} $.

On the other hand, for ${\nu  > 1}$ and $n = 1,2, \dots $, we have
\[{{a_n}({\lambda _x},{\lambda _y},\nu ) = \sum\limits_{k = 0}^n {{{\left( {\frac{{n!}}{{(n - k)!k!}}{{\left( {\frac{{{\lambda _x}}}{{{\lambda _x} + {\lambda _y}}}} \right)}^k}{{\left( {\frac{{{\lambda _y}}}{{{\lambda _x} + {\lambda _y}}}} \right)}^{n - k}}} \right)}^\nu } < {a_n}({\lambda _x},{\lambda _y},1) = 1} };\]
for ${\nu  < 1}$ and $n = 1,2, \dots $, we have
\[{{a_n}({\lambda _x},{\lambda _y},\nu ) = \sum\limits_{k = 0}^n {{{\left( {\frac{{n!}}{{(n - k)!k!}}{{\left( {\frac{{{\lambda _x}}}{{{\lambda _x} + {\lambda _y}}}} \right)}^k}{{\left( {\frac{{{\lambda _y}}}{{{\lambda _x} + {\lambda _y}}}} \right)}^{n - k}}} \right)}^\nu } > {a_n}({\lambda _x},{\lambda _y},1) = 1} }\]
The above inequality contradicts that ${a_n}({\lambda _x},{\lambda _y},\nu ) \equiv 1$ for $n = 1,2, \dots $ and $\nu  \in {{\Bbb{R}}^ + } \backslash \{ 1\} $.
\end{proof}

Now, we present the main theorem.
\begin{theorem}
  If $X_{\frac{1}{\nu}},Y_{\frac{1}{\nu}} \in RSP$ and ${{\rm{C}}_\nu }{I}_X,{{\rm{C}}_\nu }{I}_Y< \infty$, then
\begin{equation}\label{eq:five}
\frac{1}{{{\rm{C}}_\nu }{I}_{X+Y}} \geq  \frac{1}{{{\rm{C}}_\nu }{I}_X}+\frac{1}{{{\rm{C}}_\nu }{I}_Y}.
\end{equation}
The equality sign in (\ref{eq:five}) holds iff $X$ has COM-Poisson distribution (possibly shifted):
$$
  P_{X}(x)=\left(\frac{\mu^{x-x_{0}}}{(x-x_{0})!}\right)^{\nu} \bigg/  \sum\limits_{i = 0}^{\infty} (\frac{\mu ^{i}}{i!})^{\nu}, (x=x_{0},x_{0}+1,...)
$$
for some integer $z_{0}$ and $\mu>0$ as well as $Y$.
\end{theorem}

\begin{proof}
   Notice that $(X_{\nu})_{\frac{1}{\nu}}=X$ in distribution, then ${{\rm{C}}_\nu }{I}_X=\E J_{X_{\frac{1}{\nu}}}^2$ by ({\ref{eq:cvi}}). So the Kagan's characterization \label{eq:kagan} is equivalent to (\ref{eq:five}) by the result that $X$ is uniquely determined by $X_{\frac{1}{\nu}}$.
\end{proof}

Instead of that the condition distribution characterizations is related to COM-binomial distribution,  \cite{borges14} showed that the COM-Poisson distribution is the limit distribution of the COM-binomial distribution as $m$ tends to infinity.
\begin{theorem}(COM-binomial distribution approximation)
Consider $m \in \Bbb{N} $ and  $X_{m} \sim {\rm{CMB}}(m,p_{m},\nu)$, with $\nu > 0$ and $p_{m}$ such that $\lim _{m \rightarrow \infty} m^{\nu} p_{m}=\lambda$. Then, for $k = 0,1,2,...$,
$$
\lim_{m \rightarrow \infty} P(X_{m}=k)= \lim _{m \rightarrow \infty} \frac{ {{{{{m} \choose {k}}}^\nu}{p_m^k}{{(1 - p_m)}^{m - k}}} }{N(m,p_m,\nu)}
= \frac{\lambda ^{k}}{(k!)^{\nu}} \frac{1}{Z(\lambda, \nu)},
$$
namely $ \lim _{m \rightarrow \infty} X _{m} \sim {\rm{CMP}}(\lambda, \nu)$.
\end{theorem}
\cite{shmueli05} pointed out that COM-binomial distribution could be seen as a sum of equicorrelated Bernoulli r.v.'s ${\rm{ }}{Z_{i{\rm{ }}}} \ (i = 1, \dots ,m){\rm{ }}$  with joint distribution (see also \cite{borges14} )
$$
P(Z_{1}=z_{1},\dots ,Z_{m}=z_{m}) = \frac{{{{{{m} \choose {k}}}^{\nu-1}}{p^k}{{(1 - p)}^{m - k}}}}{{\sum\limits_{x_{1},...,x_{m} \in \{ 0,1 \}} {{{{{m} \choose {{x_{1}+...+x_{m}}}}}^{\nu-1}}{p^{x_{1}+...+x_{m}}}{{(1 - p)}^{m - (x_{1}+...+x_{m})}}} }}
,
$$
where $(z_{1}, \dots , z_{m}) \in \{0,1\}^{m}, k=z_{1}+\dots +z_{m}$.

\subsection{Approximate to COM-Poisson distribution}
Apart from the fact that COM-binomial approximates to COM-Poisson, the next theorem illustrates that COM-negative binomial distribution is suitable since its limiting distribution is the COM-Poisson. We show that COM-negative binomial r.v. $X \sim {\rm{CMNB}}(r,\nu ,p_{r})$ converges to the COM-Poisson r.v. $X \sim {\rm{CMP}}(\lambda,\nu )$ with $\lambda  = \lim_{r \rightarrow \infty}{r^\nu}\frac{p_{r}}{{1 - p_{r}}}$.

\begin{theorem}\label{thm:foc}
Let $X$ be a r.v. with COM-negative binomial distribution with parameters $(r,\nu ,p_{r})$ in (\ref{eq:cnbd}), denote the p.m.f. as $P(X = k)$, and let $\lambda  = \lim_{r \rightarrow \infty}{r^\nu}\frac{p_{r}}{{1 - p_{r}}}$. Then we get
$$
\mathop {\lim }\limits_{r \to  + \infty } P(X = k)=\frac{{{\lambda ^k}}}{{{{(k!)}^\nu }}} \cdot \frac{1}{{Z(\lambda ,\nu )}},(k = 0,1,2, \dots ).
$$
\end{theorem}
\begin{proof}
Notice that for large $r$, $p_{r} \approx \frac{\lambda }{{{r^\nu} + \lambda }}$, substitute it into p.m.f. (\ref{eq:cnbd}), then we can obtain
\begin{align*}
P(X=k) &\approx {\frac{{{\lambda ^k}}}{{{{(k!)}^\nu }}} \cdot {{\left( {\frac{{\Gamma (r + k)}}{{\Gamma (r)\;{r^k}}}} \right)}^\nu } \cdot \frac{1}{{{{(1 + \lambda /{r^\nu})}^k}}}  \cdot \frac{\left( {\frac{{{r^\nu}}}{{{r^\nu} + \lambda }}} \right)^r}{C(r,\nu,\frac{\lambda}{r^\nu+\lambda})}}.\\
\end{align*}

And notice that
$\mathop {\lim }\limits_{r \to  + \infty } {\left( {\frac{{\Gamma (r + k)}}{{\Gamma (r)\;{r^k}}}} \right)^\nu } \cdot \frac{1}{{{{(1 + \lambda /{r^\nu})}^k}}} = 1,(k = 0,1,2 \dots )$ and
\begin{align*}
\mathop {\lim }\limits_{r \to \infty } {{{\left( {\frac{{{r^\nu} + \lambda }}{{{r^\nu}}}} \right)}^r}C(r,\nu ,p_{r})} &= \mathop {\lim }\limits_{n \to  + \infty } \mathop {\lim }\limits_{r \to  + \infty } \sum\limits_{i = 0}^n {\frac{{{\lambda ^i}}}{{{{(i!)}^\nu }}} \cdot {{\left( {\frac{{\Gamma (r + i)}}{{\Gamma (r)\;{r^i}}}} \right)}^\nu } \cdot \frac{1}{{{{(1 + \lambda /{r^\nu})}^i}}}}  \\
&= \mathop {\lim }\limits_{n \to  + \infty } \sum\limits_{i = 0}^n {\frac{{{\lambda ^i}}}{{{{(i!)}^\nu }}}}  = Z(\lambda ,\nu ).
\end{align*}
Hence,
\[\mathop {\lim }\limits_{r \to  \infty } P(X=k)=
\frac{{{\lambda ^k}}}{{{{(k!)}^\nu }}} \cdot \frac{1}{{Z(\lambda ,\nu )}}\]
holds.
\end{proof}

\section{Other properties and functional operator characterization}

\subsection{Other properties}

A discrete r.v. $X$ obeys the discrete pseudo compound Poisson~(DPCP) distribution if its p.g.f. has the form:
\begin{equation}\label{eq:dpcp}
G_X(z) = \sum\limits_{n = 0}^\infty P(X = n)z^n  = \exp\left(\sum\limits_{k = 1}^\infty  \alpha_k \lambda (z^k - 1)\right), \quad (|z| \le 1)
\end{equation}
with parameters $(\alpha_1 \lambda,\alpha_2 \lambda, \ldots ) \in \mathbb{R}^\infty$ satisfying
\[{\sum\limits_{k = 1}^\infty  {\left| {\lambda {\alpha _k}} \right|}  < \infty ,{\alpha _k} \in {\Bbb{R}},\lambda  > 0}.\]

The explicit expression for the p.m.f. of
the DPCP distribution is given by
$$
  P_{n} = \left({\alpha _n}\lambda  +  \dots  + \sum\limits_{\scriptstyle{k_1} +  \dots  + {k_u} +  \dots {k_n} = i,{k_u} \in {\Bbb N }\hfill\atop
    \scriptstyle1\cdot{k_1} +  \dots  + u{k_u} +  \dots  + n{k_n} = n\hfill} {\frac{{\alpha _1^{{k_1}}\alpha _2^{{k_2}} \dots \alpha _n^{{k_n}}{\lambda ^i}}}{{{k_1}!{k_2}! \cdots {k_n}!}}}  +  \cdots  + \frac{{\alpha _1^n{\lambda ^n}}}{{n!}} \right) {\mathrm{e}^{ -\lambda }},(n=0,1, \dots ).
$$
For more theoretical properties of DPCP distribution, see \cite{zhang14}, \cite{zhang17} and references therein.

If all the $\alpha_k$ are non-negative,  (\ref{eq:dpcp}) is the p.g.f. of discrete compound Poisson~(DCP) distribution. Let $N$ be Poisson distributed with parameter $\lambda$ and ${Y_i}~(i=0,1, \dots )$ be i.i.d. discrete r.v.'s with $P\{ {Y_1} = k\}  = {\alpha _k}$, assure that ${Y_i}$ and $N$ are independent, then DCP distributed r.v. $X$ can be decomposed as
\[X = {Y_1} + {Y_2} +  \dots  + {Y_{N}},\]
We refer the interested reader to the survey by section 9.3 of \cite{johnson05} for references on DCP distribution and many other issues can be found in \cite{zhang14} and \cite{zhang16}.

Let $\mathcal{F}$ be a family of DCP distribution, the model must close under addition since the sum of DCP distributed r.v.'s maintains the characterization of DCP distribution. More precisely, \cite{janossy50} showed that:

Let $ {\mathcal{F_{\mu}}}$ be a family of non-negative integer-valued r.v. $X$ which can be parameterized by their mean $\mu$. The member of $ {\mathcal{F_{\mu}}}$ is closed under addition (if $X \in {{\cal F}_{{\mu _1}}},Y \in {{\cal F}_{{\mu _2}}} \Rightarrow X + Y \in {{\cal F}_{{\mu _1} + {\mu _2}}}$), where $ {\mu} $ runs over all non-negative real numbers.

Then under the conditions above, iff ${X_{\mu} }$  is discrete compound Poisson distributed with finite expectation condition.

Moreover, given a parametric model, we say that it is ``partially closed under addition" if for each r.v. $X$ belonging to this model, the sum of any number of independent copies of $X$ also belongs to this parametric model. Given a biparametric count model (mean $\mu  = \E X$ and dispersion index $d = {\rm{Var}}X/{\E X}$) which satisfies the regularity conditions:

1. Admit a change of variables so that they can be parameterized by their mean $\mu $, and dispersion index $d$.

2. The domain of the parameter $\mu$ is ${{\Bbb{R}}^ + }$. 3. Admit a p.g.f. continuous in $\mu$.

In order to be partially closed under addition, a necessary and sufficient condition is that its p.g.f. can be expressed as a discrete compound Poisson distribution, see \cite{puig06}.

With the condition of finite expectation, we can verify a distribution which is not discrete compound Poisson, if we can show that the sum of independent r.v. is not closed under addition. The mean of COM-Poisson r.v. $\E X$ can be closely approximated by the following (see \cite{shmueli05}),
\[\E X = \lambda \frac{{d\{ \ln [Z(\lambda ,\nu )]\} }}{{d\lambda }} \approx {\lambda ^{1/\nu }} - \frac{{\nu - 1}}{{2 \nu}}.\]
$\E X$ can run over all non-negative real numbers since $\lambda  \in (0,\infty)$.

Fix the $\nu~(\nu > 0)$, let $ {\mathcal{F_{\nu}}}$ be the family of COM-Poison r.v. of order $\nu$. If $\nu$ is given, then $\lambda$ is determined by $\E X$, thus $ {\mathcal{F_{\nu}}}$ is parameterized by the $\lambda$. When $\nu=1$, the well-known family of Poisson r.v.'s with mean $\lambda$ which is closed under addition. The question emerges as to whether the COM-Poisson r.v. has the same property. The answer to this question is not true except $\nu=1$.

\begin{corollary}\label{cor:cvl}
If $X,Y \in {\mathcal{F_{\nu}}}$ and $\nu  \neq 1$,  then $X+Y \notin  {\mathcal{F_{\nu}}}$ except $\nu = 1$. So COM-Poisson of order $ \nu ~(\nu \neq 1) $ is not DCP distribution.
\end{corollary}
\begin{proof}

Compared (\ref{eq:rec}) with (\ref{eq:six}), if  $X+Y \in  {\mathcal{F_{\nu}}}$, we must have $\frac{{{a_n}({\lambda _x},{\lambda _y},\nu )}}{{{a_{n - 1}}({\lambda _x},{\lambda _y},\nu )}} = f({\lambda _x},{\lambda _y}) > 0$. So ${a_n}({\lambda _x},{\lambda _y},\nu ) = f{({\lambda _x},{\lambda _y})^{(n - 1)}}$ for $n = 1,2, \dots $.

Assume $f({\lambda _x},{\lambda _y}) \ne 1$, by using p.m.f. of COM-binomial distribution,  we have
\[\begin{array}{l}
0 = \mathop {\lim }\limits_{n \to \infty } {a_n}({\lambda _x},{\lambda _y},\nu ) = \mathop {\lim }\limits_{n \to \infty } \sum\limits_{k = 0}^n {{{\left( {\frac{{n!}}{{(n - k)!k!}}{{\left( {\frac{{{\lambda _x}}}{{{\lambda _x} + {\lambda _y}}}} \right)}^k}{{\left( {\frac{{{\lambda _y}}}{{{\lambda _x} + {\lambda _y}}}} \right)}^{n - k}}} \right)}^\nu }}  \ne 0,(f({\lambda _x},{\lambda _y}) < 1)\\
\infty  = \mathop {\lim }\limits_{n \to \infty } {a_n}({\lambda _x},{\lambda _y},\nu ) = \mathop {\lim }\limits_{n \to \infty } \sum\limits_{k = 0}^n {{{\left( {\frac{{n!}}{{(n - k)!k!}}{{\left( {\frac{{{\lambda _x}}}{{{\lambda _x} + {\lambda _y}}}} \right)}^k}{{\left( {\frac{{{\lambda _y}}}{{{\lambda _x} + {\lambda _y}}}} \right)}^{n - k}}} \right)}^\nu }}  \ne \infty ,(f({\lambda _x},{\lambda _y}) > 1)
\end{array}.\]
Therefore, it leads to the contradiction.
\end{proof}

Corollary~\ref{cor:cvlt} and Corollary~\ref{cor:cvl} also provide two characterizations of Poisson distribution:

\begin{corollary}
Fix the $\nu~(\nu > 0)$, if $X \sim {\rm{CMP}}(\lambda_{2},\nu)$ and $Y \sim {\rm{CMP}}(\lambda_{2},\nu)$, then $X+Y \sim   {\rm{CMP}}(\lambda_{1}+\lambda_{2},\nu)$ iff $X,Y$ are Poisson distributed.
\end{corollary}

\begin{corollary}
Fix the $\nu~(\nu > 0)$, if $X,Y \in {\mathcal{F_{\nu}}}$,  then $X+Y \in  {\mathcal{F_{\nu}}}$  iff $X,Y$ are Poisson distributed.
\end{corollary}

The next characterization of DPCP distribution in \cite{zhang14} is a direct result from L{\'e}vy-Wiener theorem.

\begin{lemma}(L{\'e}vy-Wiener theorem)
Let $F(\theta ) = \sum\limits_{k =  - \infty }^\infty  {{c_k}{e^{ik\theta }}} ,\theta  \in [0,2\pi ]$ be a absolutely convergent Fourier series with $\left\| F \right\| = \sum\limits_{k =  - \infty }^\infty  {\left| {{c_k}} \right|}  < \infty $. The values of $F(\theta )$ lie on a curve $C$, and $H(t)$ is an analytic (not necessarily single-valued) function of a complex variable which is regular at every point of $C$. Then $H[F(\theta )]$ has an absolutely convergent Fourier series.
\end{lemma}

\begin{lemma}\label{lem:cd}
(Characterization of DPCP distribution) For any discrete r.v. $X$, its p.g.f. $G(z)$ has no zeros iff $X$ is DPCP distributed.
\end{lemma}

For $X \sim {\rm{CMP}}(\lambda ,\nu )$, is there a zero point of $\frac{Z(\lambda z ,\nu)}{Z(\lambda ,\nu)}$ when $\nu  \in (1,\infty )$? It is not sure. But we can add some other conditions to guarantee that $Z(\lambda z,\nu )$ has no zeros.

\begin{theorem}
The COM-Poisson r.v. is DPCP distributed for $\lambda \le 1$.
\end{theorem}
\begin{proof}
First, we need to show that $G(z)=\frac{Z(\lambda z ,\nu)}{Z(\lambda ,\nu)}$ has no zeros in $0 < \left| z \right| < 1$. Since ${P_k} = \frac{{{\lambda ^k}}}{{{{(k!)}^\nu }}} \cdot \frac{1}{{Z(\lambda ,\nu )}},(k = 0,1,2, \dots )$, we have
\[\left| {(1 - z)G(z)} \right| = \left| {{P_0} - ({P_0} - {P_1})z - ({P_1} - {P_2}){z^2} - \cdots  } \right| \ge {P_0} - \left| {({P_0} - {P_1})\left| z \right| + ({P_1} - {P_2})\left| {{z^2}} \right| +  \cdots } \right|,\]
If ${P_k} - {P_{k + 1}} \ge 0$ for $k =0, 1,2, \dots $, then we have
\[\frac{{{P_{k + 1}}}}{{{P_k}}} = \frac{\lambda }{{{{(k + 1)}^\nu }}} \le 1  \Rightarrow \lambda \le 1    .\]
Hence
\[{P_0} - \left| {({P_0} - {P_1})\left| z \right| + ({P_1} - {P_2})\left| {{z^2}} \right| +  \dots } \right| > {P_0} - ({P_0} - {P_1}) - ({P_1} - {P_2}) +  \dots  = 0.\]
In addition, $z = \pm 1$ is not a zero point since $G(1) = 1,G( - 1) = {P_0} - {P_1} + {P_2} - {P_3} +  \cdots  > 0$.
\end{proof}

For $\lambda \le 1$, rewrite the p.g.f. of COM-Poisson distribution as the expression (\ref{eq:dpcp}). Use the recurrence relation (L{\'e}vy-Adelson-Panjer recursion) of p.m.f. of DPCP distribution, see \cite{buchmann03} and Remark 1 in \cite{zhang14}, we have
$$
  {P_{n + 1}} = \frac{{\tilde \lambda (\lambda ,\nu )}}{{n + 1}}[{\alpha _1}(\lambda ,\nu ){P_n} + 2{\alpha _2}(\lambda ,\nu ){P_{n - 1}} +  \cdots  + (n + 1){\alpha _{n + 1}}(\lambda ,\nu ){P_0}],({P_0} = {e^{ - \tilde \lambda (\lambda ,\nu )}},n = 0,1, \cdots ).
$$
where
\[\frac{{Z(\lambda z,\nu )}}{{Z(\lambda ,\nu )}} = \exp \left( {\sum\limits_{k = 1}^\infty  {\tilde \lambda (\lambda ,\nu ){\alpha _k}(k,\lambda )({z^k} - 1)} } \right) \buildrel \Delta \over = \sum\limits_{i = 0}^\infty  {{P_i}{z^i}}. \]

In this case, the recurrence relation of COM-Poisson distribution has the alternative form with infinite terms. Notice that ${P_k} = P(X = k) = \frac{{{\lambda ^k}}}{{{{(k!)}^\nu }}} \cdot \frac{1}{{Z(\lambda ,\nu )}}$, then the parameters of DPCP distribution of COM-Poisson distribution are determined by the following system of equations:
\[\frac{{{\lambda ^{n + 1}}}}{{{{((n + 1)!)}^\nu }}} = \frac{{\tilde \lambda (\lambda ,\nu )}}{{n + 1}}[{\alpha _1}(\lambda ,\nu )\frac{{{\lambda ^n}}}{{{{(n!)}^\nu }}} + 2{\alpha _2}(\lambda ,\nu )\frac{{{\lambda ^{n - 1}}}}{{{{((n - 1)!)}^\nu }}} +  \cdots  + (n + 1){\alpha _{n + 1}}(\lambda ,\nu )],(n = 0,1, \cdots ),\]
where ${\tilde \lambda (\lambda ,\nu )} = \ln {P_0}$.

A limit case as $\nu  \to  + \infty $, $Z(\lambda ,\nu ) \to 1 + \lambda $, hence COM-Poisson distribution tends to a Bernoulli distribution with $P(X = 0) = \frac{1}{{1 + \lambda }}$. The p.g.f. of this Bernoulli distribution $\frac{{z + \lambda z}}{{1 + \lambda }}$ has zero point if $\lambda  \ge 1$.

\subsection{Functional operator characterization}
The recurrence relation (\ref{eq:rec}) helps us consider the following functional operator characterization, which is well-known in
the Stein-Chen method literature, or from the work of \cite{brown01} who studied the functional operator characterizations (we call it Stein identity) a very large class of the equilibrium distribution of a birth-death process. Here we want
to give an alternative proof, see Appendix for the proof.

\begin{lemma}
  Let $g: \Bbb{N} \to \Bbb{R}$ be a bounded function, suppose there
  exists a bounded solution $f:{\Bbb N} \to {\Bbb R}$ satisfies
$$
{w^\nu }f(w) - \lambda f(w + 1) = g(w),\quad (\nu  > 0, \lambda  > 0).
$$
Then $W$ is COM-Poisson distributed if and only if $\E g(W)=0$ provided that $\E {W^\nu }f(W)<\infty $.
\end{lemma}

By using the Stein identity lemma above and the Stein-Chen method, \cite{daly16} have shown some convergence results and approximations, including a bound on the total variation distance between a COM-binomial r.v. and the corresponding COM-Poisson r.v..

\section{Conclusion}
The results in this works contain some theoretical properties of the COM-Poison distribution, with potentially researches in the probability theory. For example, the COM-type distribution as a ``$\nu$-power distribution" could be extended to continuous distributions, and more probability properties could be explored. In statistics, goodness-of-fit tests of COM-Poisson distribution based on the new characterizations (Stein identity, conditional distribution) shed light on more powerful tests than some omnibus goodness-fit tests like Kolmogorov-Smirnov or Chi-squared test.

\section{Acknowledgments}
This work was partially supported by the National Natural Science Foundation of China (No.11201165). The authors thank Arash Sioofy Khoojine for careful reading and insightful comments, which
markedly improve the quality of this paper.

\section{Appendix}

\subsection{Queuing systems characterization}
The COM-Poisson distribution can be generated as a queueing system with Poisson arrival rate $\lambda_{0}$ and exponential service rate $\mu_{n}=\mu n^{\nu}$ (which depends on the size of the queue $n$). Let $P_{n}(t)$ be the size of $n$ in the queue at time $t$ , and it also satisfies the system of differential difference equations (due to \cite{conway62}):

\begin{numcases}{}
P_{0}(t+\Delta t)=(1-\lambda_{0} \Delta t)P_{0}(t)+ \mu \cdot 1^{\nu} \Delta t P_{1}(t)\label{eq:one1}\\
P_{0}(t+\Delta t)=\lambda_{0} \Delta t P_{n-1}(t)+ (1-\lambda_{0} \Delta t - \mu n^{\nu} \Delta t) P_{n}(t) + \mu (n+1)^{\nu} \Delta t P_{n+1}(t), (n=1,2,...)\label{eq:one2}
\end{numcases}

From (\ref{eq:one1}) , we have $P_{0}^{'}(t)=-\lambda_{0} P_{0}(t)+\mu P_{1}(t)$.

Suppose that the queue reaches a steady state as $t\rightarrow \infty$, then $ \lim_{t\rightarrow \infty} P_{n}^{'}(t)=0$.
Set $P_{n}=\lim _{t\rightarrow \infty } P_{n}(t)$, we obtain $P_{1}=\frac{\lambda_{0}}{\mu} P_{0}$.

Using the approach above again, (\ref{eq:one2}) becomes
\begin{equation} \label{eq:three}
  0= -(\lambda_{0} +\mu n^{\nu}) P_{n} + \lambda_{0} P_{n-1} + \mu (n+1)^{\nu} P_{n+1}(t), (n=1,2,...)
\end{equation}
which implies $P_{n+1}= (\frac{\lambda_{0}}{\mu}+n^{\nu} ) \frac{P_{n}}{(n+1)^{\nu}} - \frac{\lambda_{0}}{\mu}  \frac{P_{n-1}}{(n+1)^{\nu}}   $.

Let $n=1$, we have
$$
  P_{2}(t)=(\frac{\lambda_{0}}{\mu}+1) \frac{\lambda_{0}}{\mu} \frac{P_{0}}{2^{\nu}}-\frac{\lambda_{0}}{\mu} \frac{P_{0}}{2^{\nu}}=(\frac{\lambda_{0}}{\mu})^{2} \frac{P_0}{2^{\nu}}.
$$

Let $n=2$, we have
$$
  P_{3}(t)=(\frac{\lambda_{0}}{\mu}+ 2^{\nu})  \frac{P_{2}}{3^{\nu}}-\frac{\lambda_{0}}{\mu} \frac{P_{1}}{3^{\nu}}=(\frac{\lambda_{0}}{\mu})^{3} \frac{P_0}{3^{\nu}}.
$$

In general, applying (\ref{eq:three}) ,we can get the follow formula by induction:
$$
  P_{n}=(\frac{\lambda_{0}}{\mu})^{n} \frac{P_{0}}{(n!)^{\nu}} .
$$

\subsection{Proof of Lemma 4.3}

\begin{proof}
  Sufficiency: If the r.v. $W \sim {\rm{CMP}}(\lambda ,\nu )$, then we have
 \begin{align*}
    \lambda \E [f(W+1)] = \lambda \sum\limits_{\omega=0}^{\infty}[f(\omega+1)P(W=\omega)]
    &=\sum\limits_{\omega=0}^{\infty}\frac{\lambda^{\omega+1}}{(\omega !)^{\nu}}\frac{1}{Z(\lambda,\nu)}f(\omega+1)\\
    &=\sum\limits_{\omega +1=1}^{\infty}\frac{\lambda^{\omega+1}}{((\omega+1)!)^{\nu}}\frac{1}{Z(\lambda,\nu)}(\omega+1)^{\nu}f(\omega+1)\\
    &=\E [W^{\nu}f(W)].
\end{align*}
Hence $\E g(W) = 0$.

Necessity: If $\E g(W) = 0$ holds for all bounded functions. Without loss of
generality, for $j\in \Bbb{N}$, we could take $f(\omega) = 1[\omega =j]$ where $1[\omega =j]$ is the indicator function, so that the $\E g(W) = 0$  implies that
$$
  \lambda \E [f(W+1)] =\E [W^{\nu}f(W)].
$$
Hence,
$$
  \lambda P(W=x-1)= x^{\nu}P(W=x).
$$
So $W$ has the COM-Poisson distribution from Lemma \ref{lem:rec}, i.e., the recurrence relation characterization.
\end{proof}


\section{References}

\begin{thebibliography}{11}
\providecommand{\natexlab}[1]{#1} \providecommand{\url}[1]{\texttt{#1}} \providecommand{\urlprefix}{URL }
\expandafter\ifx\csname urlstyle\endcsname\relax
  \providecommand{\doi}[1]{doi:\discretionary{}{}{}#1}\else
  \providecommand{\doi}{doi:\discretionary{}{}{}\begingroup
  \urlstyle{rm}\Url}\fi
\providecommand{\eprint}[2][]{\url{#2}}

\bibitem[{Ahmad(2007)}]{ahmad07}
Ahmad, M. (2007). A short note on Conway-Maxwell-hyper Poisson distribution. Pakistan Journal of Statistics, 23(2), 135-137.

\bibitem[{Borges et al.(2014)}]{borges14}
Borges, P., Rodrigues, J., Balakrishnan, N., Baz¨¢n, J. (2014). A COM-Poisson type generalization of the binomial distribution and its properties and applications. Statistics \& Probability Letters, 87, 158-166.

\bibitem[{Brown and Xia(2001)}]{brown01}
Brown, T. C., Xia, A. (2001). Stein's method and birth-death processes. The Annals of Probability, 29(3), 1373-1403.

\bibitem[{Buchmann(2003)}]{buchmann03}
Buchmann, B., Gr{\"u}bel, R. (2003). Decompounding: an estimation problem for Poisson random sums. The Annals of Statistics, 31(4), 1054--1074.

\bibitem[{Chatterji(1963)}]{chatterji63}
Chatterji, S. D. (1963). Some elementary characterizations of the Poisson distribution. American Mathematical Monthly, 70(9), 958-964.

\bibitem[{Chakraborty and Imoto(2016)}]{chakraborty15}
Chakraborty, S., Imoto, T. (2016). Extended Conway-Maxwell-Poisson distribution and its properties and applications. Journal of Statistical Distributions and Applications, 3(1), 5.

\bibitem[{Conway and Maxwell(1962)}]{conway62}
Conway, R. W., Maxwell, W. L. (1962). A queuing model with state dependent service rates. Journal of Industrial Engineering, 12(2), 132-136.

\bibitem[{Daly and Gaunt(2016)}]{daly16}
Daly, F., Gaunt, R. E. (2016). The Conway-Maxwell-Poisson distribution: distributional theory and approximation. ALEA: Latin American Journal of Probability and Mathematical Statistics,13, 635-658.

\bibitem[{Gillispie and Green(2015)}]{gillispie15}
Gillispie, S. B., Green, C. G. (2015). Approximating the Conway-Maxwell-Poisson distribution normalization constant. Statistics, 49(5), 1062-1073.

\bibitem[{G{\'o}mez-D{\'e}niz and Calder{\'i}n-Ojeda(2014)}]{gomez14}
G{\'o}mez-D{\'e}niz, E., Calder{\'i}n-Ojeda, E. (2014). Unconditional distributions obtained from conditional specification models with applications in risk theory. Scandinavian Actuarial Journal, 2014(7), 602-619.

\bibitem[{Haight(1967)}]{haight67}
Haight, F. A. (1967). Handbook of the Poisson distribution, Wiley, Los Angeles.

\bibitem[{Imoto(2014)}]{imoto14}
Imoto, T. (2014). A generalized Conway-Maxwell-Poisson distribution which includes the negative binomial distribution. Applied Mathematics and Computation, 247, 824-834.

\bibitem[{J{\'a}nossy et al.(1950)}]{janossy50}
J{\'a}nossy, L., R{\'e}nyi, A., Acz{\'e}l, J. (1950). On composed Poisson distributions, I. Acta Mathematica Hungarica, 1(2), 209--224.

\bibitem[{Johnson et al.(2005)}]{johnson05}
Johnson, N. L., Kemp, A. W., Kotz S. (2005). Univariate Discrete Distributions, 3ed. Wiley, New Jersey.

\bibitem[{Kagan(2001)}]{kagan01}
Kagan, A. (2001). A discrete version of the Stam inequality and a characterization of the Poisson distribution. Journal of statistical planning and inference, 92(1), 7-12.

\bibitem[{Kagan et al.(1973)}]{kagan73}
Kagan, A. M., Rao, C. R., Linnik, Y. V. (1973). Characterization problems in mathematical statistics, Wiley.

\bibitem[{Kokonendji et al.(2008)}]{kokonendji08}
Kokonendji, C. C., Mizere, D., Balakrishnan, N. (2008). Connections of the Poisson weight function to overdispersion and underdispersion. Journal of Statistical Planning and Inference, 138(5), 1287-1296.

\bibitem[{Lin and Hu(2008)}]{lin08}
Lin, G. D., Hu, C. Y. (2001). The Riemann zeta distribution. Bernoulli, 7(5), 817-828.

\bibitem[{Moran(1952)}]{moran52}
Moran, P. A. P. (1952). A characteristic property of the Poisson distribution. In Mathematical Proceedings of the Cambridge Philosophical Society (Vol. 48, No. 01, pp. 206-207). Cambridge University Press.

\bibitem[{Nadarajah(2009)}]{nadarajah09}
Nadarajah, S. (2009). Useful moment and CDF formulations for the COM-Poisson distribution. Statistical Papers, 50(3), 617-622.

\bibitem[{Patil et al.(1975)}]{patil75}
Patil, P.~(Editor), Kotz,S.~(Editor), Ord, J.K.~(Editor). (1975). A Modern Course on Statistical Distributions in Scientific Work: Volume 3 - Characterizations and Applications. Springer.

\bibitem[{Patil and Seshadri(1964)}]{patil64}
Patil, G. P., Seshadri, V. (1964). Characterization theorems for some univariate probability distributions. Journal of the Royal Statistical Society. Series B (Methodological), 286-292.

\bibitem[{Puig and Valero(2006)}]{puig06}
Puig, P.,  Valero, J. (2006). Count data distributions: some characterizations with applications. Journal of the American Statistical Association, 101(473), 332-340.

\bibitem[{Pog{\'a}ny(2016)}]{pogany16}
Pog{\'a}ny, T. K. (2016). Integral form of the COM-Poisson renormalization constant. Statistics \& Probability Letters, 119, 144-145.

\bibitem[{Rao and Rubin(1964)}]{rao64}
Rao, C. R., Rubin, H. (1964). On a characterization of the Poisson distribution. Sankhy{\=a}: The Indian Journal of Statistics, Series A, 295-298.

\bibitem[{R{\'e}nyi(1961)}]{renyi61}
R{\'e}nyi, A. (1961). On measures of entropy and information. Proceedings of the Fourth Berkeley symposium on mathematical statistics and probability (Vol. 1, pp. 547-561).

\bibitem[{Shanbhag(1977)}]{shanbhag77}
Shanbhag, D. N. (1977). An extension of the Rao-Rubin characterization of the Poisson distribution. Journal of Applied Probability, 14(3), 640-646.

\bibitem[{Shmueli et al.(2005)}]{shmueli05}
Shmueli, G., Minka, T. P., Kadane, J. B., Borle, S., Boatwright, P. (2005). A useful distribution for fitting discrete data: revival of the Conway-Maxwell-Poisson distribution. Journal of the Royal Statistical Society: Series C (Applied Statistics), 54(1), 127-142.

\bibitem[{Stam(1959)}]{stam59}
Stam, A. J. (1959). Some inequalities satisfied by the quantities of information of Fisher and Shannon. Information and Control, 2(2), 101-112.

\bibitem[{Tsallis(1988)}]{tsallis88}
Tsallis, C. (1988). Possible generalization of Boltzmann-Gibbs statistics. Journal of statistical physics, 52(1-2), 479-487.

\bibitem[{Wang(1975)}]{wang75}
Wang, P. C. (1975). Characterizations of the Poisson distribution based on random splitting and random expanding. Discrete Mathematics, 13(1), 85-93.

\bibitem[{Zhang et al.(2014)}]{zhang14}
Zhang, H., Liu, Y., Li, B. (2014). Notes on discrete compound Poisson model with applications to risk theory. Insurance: Mathematics and Economics, 59, 325-336.

\bibitem[{Zhang and Li(2016)}]{zhang16}
Zhang, H., Li, B. (2016). Characterizations of discrete compound Poisson distributions. Communications in Statistics-Theory and Methods, 45(22), 6789-6802.

\bibitem[{Zhang et al.(2017)}]{zhang17}
Zhang, H., Li, B., \& Kerns, G. J. (2017). A characterization of signed discrete infinitely divisible distributions. Studia Scientiarum Mathematicarum Hungarica, 54(4), 446-470.

\bibitem[{Zhang et al.(2018)}]{zhang18}
Zhang, H., Tan, K., \& Li, B. (2018). COM-negative binomial distribution: modeling overdispersion and ultrahigh zero-inflated count data. Frontiers of Mathematics in China, 13(4), 967-998.

\end{thebibliography}
\end{document}